\documentclass[12pt]{article}

\usepackage[margin=1in]{geometry}

\usepackage{amsmath}
\usepackage{amssymb}
\usepackage{amsthm}
\usepackage{mathtools}
\usepackage{setspace}
\usepackage[authoryear,round]{natbib}

\usepackage{bbm}
\usepackage{booktabs}
\usepackage{graphicx}

\usepackage[hidelinks]{hyperref}

\graphicspath{{img/}}

\newtheorem{theorem}{Theorem}
\newtheorem{proposition}{Proposition}

\newtheorem{lemma}{Lemma}
\newtheorem{corollary}{Corollary}
\newtheorem{assumption}{Assumption}

\theoremstyle{definition}
\newtheorem{definition}{Definition}

\DeclareMathOperator{\rank}{rank}
\DeclareMathOperator{\tr}{tr}
\DeclareMathOperator{\rad}{rad}
\DeclareMathOperator{\supp}{supp}

\begin{document}
	\doublespacing
	
	\title{Drift Estimation under Low-Rank-plus-Sparse Structure for High-Dimensional L\'evy-Driven Ornstein--Uhlenbeck Processes}
	
	\author{Marina Palaisti\\Huron University College, Western University\\1349 Western Road\\London, Ontario N6G 1H3, Canada\\\texttt{mpalaist@uwo.ca}}
	
	\date{}
	
	\maketitle
	\begin{center}
		\textbf{Running title:} Low-Rank-plus-Sparse L\'evy--OU Drift
	\end{center}
	
	\begin{abstract}
		We study drift estimation for discretely observed high-dimensional L\'evy-driven Ornstein--Uhlenbeck processes when the drift admits a low-rank-plus-sparse approximation. A localized, increment-truncated pseudo-likelihood is regularized by nuclear and entrywise $\ell_1$ norms. A weighted joint-cone argument and restricted cancellation condition yield non-asymptotic Frobenius oracle inequalities for exact and approximate structure. The stochastic complexity is $rd+s\log d$, with explicit discretization, truncation, and L\'evy-regime sample-complexity terms. We derive operator- and infinity-norm score bounds from a Gaussian-width inequality and give a proximal-gradient algorithm. In 100-replicate experiments with mixed low-rank-plus-sparse drift, the estimator reduces mean relative Frobenius error by 4.4--5.9\% versus localized Lasso and wins 89--92\% of paired paths. Component diagnostics show that these gains concern recovery of the total drift rather than exact identification of its low-rank and sparse components. \end{abstract}
	
	\medskip
	
	
	\medskip
	
	
	\section{Introduction}
	
	High-dimensional Ornstein--Uhlenbeck (OU) processes provide continuous-time models for multivariate mean-reverting systems in which the drift matrix encodes linear dependence among coordinates. When the background driving process is L\'evy, the model additionally permits jumps, pure-jump dynamics, anisotropic shocks, and heavy tails. Sparse drift estimation has been developed for diffusion-type OU models \citep{GaiffasMatulewicz2019,CiolekMarushkevychPodolskij2020} and for L\'evy-driven OU processes under continuous and discrete observation schemes \citep{DexheimerStrauch2024,DexheimerJeszka2026}.
	
	Pure sparsity can be restrictive when dependence contains global common modes together with localized direct interactions. We therefore study drift matrices that can be represented or approximated by
	\[A_0=L_0+S_0,\]
	where $L_0$ has low effective rank and $S_0$ is sparse. Such decompositions are standard in noisy matrix decomposition, robust principal components, and latent-variable modelling \citep{ChandrasekaranSanghaviParriloWillsky2011,AgarwalNegahbanWainwright2012}. The statistical issue is not merely to combine a nuclear norm with an $\ell_1$ norm: the low-rank and sparse geometries can overlap, so the analysis must explicitly control cancellation between component errors.
	
	In high-dimensional mean-reverting financial systems, this decomposition has a direct applied interpretation. The sparse component $S_0$ can represent a relatively small set of direct bilateral or sector-specific interactions, while the low-rank component $L_0$ captures market-wide or sectoral common modes that generate dense dependence across the system. L\'evy innovations additionally accommodate abrupt shocks that are not well represented by purely Gaussian diffusion models. In such settings, estimation of the total drift is relevant for forecasting mean reversion and assessing the propagation of shocks through the system.
	
	We observe the stationary solution of
	\begin{equation}
		dX_t=-A_0X_t\,dt+dZ_t,\qquad t\ge0, \label{eq:ou}
	\end{equation}
	at $t_k=k\Delta_n$, $k=0,\ldots,n$, over horizon $T=n\Delta_n$. Following the discrete L\'evy--OU framework of \citet{DexheimerJeszka2026}, observations are localized to a bounded set and increments are truncated. The proposed estimator is
	\begin{equation}
		(\widehat L,\widehat S)\in \arg\min_{L,S\in\mathbb R^{d\times d}} \left\{ \mathcal L_n^D(L+S)+\lambda_*\|L\|_*+\lambda_1\|S\|_1 \right\},
		\qquad \widehat A=\widehat L+\widehat S, \label{eq:estimator_intro}
	\end{equation}
	where $\mathcal L_n^D$ is the localized and increment-truncated pseudo-likelihood.
	
	The analysis has four main ingredients. First, the penalized basic inequality yields one \emph{weighted joint cone}; it does not yield independent low-rank and sparse cones. We retain that joint geometry throughout. Second, structural identifiability is expressed through a restricted cancellation coefficient on this cone. Third, the stochastic score is controlled in both operator and entrywise maximum norms. Applying the Gaussian-width inequality of \citet[Theorem~4.8]{DexheimerJeszka2026} to rank-one and coordinate indexing sets produces the dimension factors $\sqrt d$ and $\sqrt{\log(d^2)}$, respectively. Hence the low-rank stochastic contribution is proportional to $rd/T$, not $r\log d/T$. Fourth, the discretization and truncation terms are kept explicit and are tied to the same high-frequency and tail regimes as the underlying discrete L\'evy--OU theory.
	
	For exact rank $r$ and sparsity $s$, the resulting high-probability Frobenius bound has the schematic form
	\begin{equation}
		\|\widehat A-A_0\|_F^2 \lesssim d^2\Delta_n^2+\frac{d}{T} +\frac{\gamma(\Delta_n)}{T} \left\{rd+s\log d\right\},
		\label{eq:intro_rate}
	\end{equation}
	under the standard truncation regimes and uniformly well-conditioned localized covariance. The factor $\gamma(\Delta_n)$ is given explicitly in \eqref{eq:gamma}. We also prove an oracle version of the result in which $A_0$ need only be well approximated by a rank-$r$ plus $s$-sparse matrix. This introduces an approximation term rather than requiring exact structure.
	
	The paper also supplies the computational side of the estimator. Product-space proximal gradient alternates singular-value and entrywise soft thresholding, while chronological validation selects an overall regularization scale and the relative weight of the nuclear penalty. A 100-replicate Monte Carlo study compares the combined estimator with localized Lasso for $d=30$ under Brownian-plus-compound-Poisson noise. For a fixed sampling mesh $\Delta_n=0.05$, the mixed low-rank-plus-sparse design is studied at horizons $T=50,100,200$, together with a sparse-only benchmark at $T=100$. The combined estimator lowers mean relative Frobenius error in every design and wins 89--92\% of paired paths in the mixed cases. Additional diagnostics distinguish total-drift estimation from component identification: sparse-component error decreases with the horizon, whereas numerical rank and low-rank subspace recovery remain substantially less stable.
	
	\section{Model, observation scheme, and estimator}\label{sec:model}
	
	\subsection{Stationary L\'evy-driven OU model}
	
	For a matrix $M$, let $\|M\|_F$, $\|M\|_{\mathrm{op}}$, $\|M\|_*$, $\|M\|_1$, and $\|M\|_\infty$ denote the Frobenius, operator, nuclear, entrywise $\ell_1$, and entrywise maximum norms. Define
	\[\mathcal M_+(\mathbb R^d):=\left\{A\in\mathbb R^{d\times d}:\Re(\lambda)>0
	\text{ for every eigenvalue }\lambda\text{ of }A\right\}.\]
	Equivalently, $A\in\mathcal M_+(\mathbb R^d)$ if $\|e^{-tA}\|_{\mathrm{op}}\to0$ as $t\to\infty$.
	
	Let $Z$ be a $d$-dimensional L\'evy process with generating triplet $(\beta,C,\nu)$, and assume that $Z$ is a square-integrable martingale. Write
	\begin{equation}
		\nu_2:=\int_{\mathbb R^d}zz^\top\,\nu(dz).\label{eq:nu2}
	\end{equation}
	We consider \eqref{eq:ou} with $A_0\in\mathcal M_+(\mathbb R^d)$. Square integrability gives the logarithmic moment condition needed for existence of an invariant law, and the standard OU--L\'evy theory gives a unique invariant distribution $\pi$; see \citep{SatoYamazato1984,Masuda2004}. We assume $X_0\sim\pi$ and that $X_0$ is independent of $Z$. Under the polynomial moment condition imposed below, $X$ is exponentially $\beta$-mixing \citep{Masuda2004}.
	
	We observe $X$ at $t_k=k\Delta_n$, $k=0,\ldots,n$, and write
	\[X_{k-1}:=X_{t_{k-1}},\qquad \Delta X_k:=X_{t_k}-X_{t_{k-1}},\qquad T=n\Delta_n.\]
	Let $B\subset\mathbb R^d$ be a bounded localization set of radius $b$, and let $\eta>0$ be an increment threshold. Define
	\begin{equation}
		I_k:=\mathbbm 1_{\{X_{k-1}\in B\}} \mathbbm 1_{\{\|\Delta X_k\|\le\eta\}}.
		\label{eq:filter}
	\end{equation}
	The localized and increment-truncated empirical covariance is
	\begin{equation}
		\widehat C_n^{B,\eta}:=\frac{1}{n}\sum_{k=1}^n I_kX_{k-1}X_{k-1}^\top. \label{eq:Chat}
	\end{equation}
	Its untruncated localized population target is
	\begin{equation}
		C_\infty(B):= \mathbb E_\pi\left[X_0X_0^\top\mathbbm1_{\{X_0\in B\}}\right].
		\label{eq:Cinf}
	\end{equation}
	The relation between \eqref{eq:Chat} and \eqref{eq:Cinf} is handled through the increment-truncated covariance concentration result recalled in Section~\ref{sec:probabilistic}; no equality of their expectations is assumed.
	
	\begin{assumption}[OU/L\'evy regularity]\label{ass:A0}
		The following conditions hold.
		\begin{enumerate}
			\item $A_0\in\mathcal M_+(\mathbb R^d)$ and $Z$ is a square-integrable L\'evy martingale whose L\'evy measure admits a $p$th moment for some $p>2$.
			\item $X_0\sim\pi$ is independent of $Z$.
			\item $\lambda_{\min}(C_\infty(B))\ge c_0>0$ and $\lambda_{\max}(C_\infty(B))\le C_0<\infty$ for the localization sequence under consideration.
		\end{enumerate}
	\end{assumption}
	
	\subsection{Localized pseudo-likelihood}
	
	We use the discrete pseudo-likelihood normalization of \citet{DexheimerJeszka2026}:
	\begin{equation}
		\mathcal L_n^D(A) :=\frac{1}{T}\sum_{k=1}^n I_k\langle AX_{k-1},\Delta X_k\rangle+\frac{\Delta_n}{2T}\sum_{k=1}^n I_k\|AX_{k-1}\|_2^2. \label{eq:pseudoloss}
	\end{equation}
	Up to an additive term independent of $A$ and multiplication by a positive scalar, minimizing \eqref{eq:pseudoloss} is equivalent to minimizing the filtered quadratic residual
	\[\sum_{k=1}^n I_k\|\Delta X_k+\Delta_n AX_{k-1}\|_2^2.\]
	The estimator is
	\begin{equation}
		(\widehat L,\widehat S) \in\arg\min_{L,S} \left\{\mathcal L_n^D(L+S)+\lambda_*\|L\|_*+\lambda_1\|S\|_1\right\}, \qquad \widehat A=\widehat L+\widehat S. \label{eq:estimator}
	\end{equation}
	
	For any $H\in\mathbb R^{d\times d}$, the loss has the exact expansion
	\begin{equation}
		\mathcal L_n^D(A_0+H)-\mathcal L_n^D(A_0)-\langle G_n,H\rangle=\frac{1}{2} \tr\!\left(H\widehat C_n^{B,\eta}H^\top\right), \label{eq:exactexpansion} \end{equation}
	where
	\begin{equation}
		G_n:=\nabla\mathcal L_n^D(A_0)=\frac{1}{T}\sum_{k=1}^n I_k \bigl(\Delta X_k+\Delta_nA_0X_{k-1}\bigr)X_{k-1}^\top. \label{eq:score}
	\end{equation}
	Thus all non-quadratic effects enter through the score; there is no unspecified second-order ``bias term.''
	
	\section{Joint low-rank--sparse geometry} \label{sec:geometry}
	
	Let a comparator $\bar A=\bar L+\bar S$ be given with $\rank(\bar L)\le r$ and $\|\bar S\|_0\le s$. Write a compact SVD $\bar L=U\Sigma V^\top$, and define the tangent and support spaces
	\begin{align}
		\mathcal T(\bar L)&:=\{UY^\top+XV^\top:X,Y\in\mathbb R^{d\times r}\},\label{eq:tangent}\\
		\mathcal S_\Omega&:=\{M\in\mathbb R^{d\times d}:\supp(M)\subseteq\Omega\},
		\qquad \Omega:=\supp(\bar S).\label{eq:support}
	\end{align}
	Let $P_\mathcal T$, $P_{\mathcal T^\perp}$, $P_\Omega$, and $P_{\Omega^c}$ denote the corresponding orthogonal projections.
	
	For $\rho>0$, define the weighted joint cone
	\begin{equation}
		\mathcal C_4(\rho;\bar L,\bar S) :=\left\{(H_L,H_S): \|P_{\mathcal T^\perp}H_L\|_*+\rho\|P_{\Omega^c}H_S\|_1 \le4\left( \|P_\mathcal T H_L\|_*+\rho\|P_\Omega H_S\|_1 	\right)\right\}.
		\label{eq:jointcone}
	\end{equation}
	The cone is joint. In particular, membership in \eqref{eq:jointcone} does not assert either of the two corresponding componentwise cone inequalities.
	
	\begin{definition}\label{def:mu}
		The \emph{restricted cancellation coefficient} associated with
		$\mathcal C_4(\rho;\bar L,\bar S)$ is
		\begin{equation}
			\mu_\rho(\bar L,\bar S)
			:=
			\sup
			\frac{|\langle H_L,H_S\rangle|}
			{\|H_L\|_F\|H_S\|_F},
			\label{eq:mu}
		\end{equation}
		where the supremum is taken over
		$(H_L,H_S)\in\mathcal C_4(\rho;\bar L,\bar S)$
		with both components nonzero.
	\end{definition}
	
	\begin{assumption}[Restricted non-cancellation]\label{ass:angle}
		For $\rho=\lambda_1/\lambda_*$, there exists $\bar\mu<1$ such that $\mu_\rho(\bar L,\bar S)\le\bar\mu$ for every comparator used in the oracle bound. Set 
		\begin{equation}
			\alpha:=\sqrt{1-\bar\mu}>0. \label{eq:alpha}
		\end{equation}
	\end{assumption}
	
	\begin{proposition}\label{prop:angle}
		Under Assumption~\ref{ass:angle}, every $(H_L,H_S)\in\mathcal C_4(\rho;\bar L,\bar S)$ satisfies
		\begin{equation}
			\|H_L+H_S\|_F\ge\alpha\bigl(\|H_L\|_F^2+\|H_S\|_F^2\bigr)^{1/2}.
			\label{eq:compatibility}	\end{equation}
	\end{proposition}
	
	\begin{proof}
		Let $x=\|H_L\|_F$ and $y=\|H_S\|_F$. By Definition~\ref{def:mu},
		\[\|H_L+H_S\|_F^2	=x^2+y^2+2\langle H_L,H_S\rangle \ge x^2+y^2-2\bar\mu xy
		\ge(1-\bar\mu)(x^2+y^2), \]where the last step uses $2xy\le x^2+y^2$.
	\end{proof}
	
	A familiar first-order diagnostic is the tangent--support angle	\begin{equation}
		\zeta(\bar L,\Omega):=\|P_\Omega P_{\mathcal T(\bar L)}\|_{F\to F}. \label{eq:zeta}
	\end{equation}
	For $H_L\in\mathcal T(\bar L)$ and $H_S\in\mathcal S_\Omega$, $|\langle H_L,H_S\rangle|\le\zeta\|H_L\|_F\|H_S\|_F$. Thus $\zeta<1$ rules out exact first-order overlap of the tangent and support spaces. Assumption~\ref{ass:angle} extends this transversality to the full weighted error cone, where off-tangent and off-support components generated by penalization must also be controlled. This is the precise geometric role of rank--sparsity incoherence in the present analysis; compare \citep{ChandrasekaranSanghaviParriloWillsky2011,AgarwalNegahbanWainwright2012}.
	
	\section{A deterministic oracle inequality} \label{sec:deterministic}
	
	Define the empirical quadratic norm
	\begin{equation}
		\|H\|_n^2:=\tr\!\left(H\widehat C_n^{B,\eta}H^\top\right).
		\label{eq:empnorm}
	\end{equation}
	We isolate the stochastic score from the discretization and truncation effects by writing
	\begin{equation}
		G_n=M_n+R_n.
		\label{eq:score_split}
	\end{equation}
	The probabilistic section constructs this decomposition explicitly and verifies the following event.
	
	\begin{definition}\label{def:good}
		The \emph{good event} $\mathcal E(\kappa,K,B_n)$, for constants $0<\kappa\le K$, penalty levels $(\lambda_*,\lambda_1)$, and $B_n\ge0$, is the event on which
		\begin{align}
			\kappa I_d\preceq\widehat C_n^{B,\eta}\preceq K I_d,\label{eq:curvature_event}\\
			\|M_n\|_{\mathrm{op}}\le\lambda_*/2,\qquad \|M_n\|_\infty\le\lambda_1/2,\label{eq:score_event}
		\end{align}
		and, uniformly for all $H\in\mathbb R^{d\times d}$,
		\begin{equation}
			|\langle R_n,H\rangle|\le\frac18\|H\|_n^2+B_n.\label{eq:remainder_event}
		\end{equation}
	\end{definition}

	\begin{theorem}\label{thm:det_exact}
		Suppose $A_0=L_0+S_0$, with $\rank(L_0)\le r$ and $\|S_0\|_0\le s$, and suppose Assumption~\ref{ass:angle} holds for $(L_0,S_0)$. On $\mathcal E(\kappa,K,B_n)$, every solution of \eqref{eq:estimator} satisfies
		\begin{equation}
			\|\widehat A-A_0\|_F^2 	\le \frac{56}{3\kappa}B_n
			+\frac{196}{9\kappa^2\alpha^2} \left(2r\lambda_*^2+s\lambda_1^2\right).
			\label{eq:det_exact}
		\end{equation}
	\end{theorem}
	
	\begin{proof}
		Set
		\[H_L=\widehat L-L_0,\qquad H_S=\widehat S-S_0,\qquad H=H_L+H_S. \]
		Define
		\begin{align*}
			P_T&:=\lambda_*\|P_\mathcal T H_L\|_*+\lambda_1\|P_\Omega H_S\|_1,\\ P_\perp&:=\lambda_*\|P_{\mathcal T^\perp}H_L\|_*+\lambda_1\|P_{\Omega^c}H_S\|_1.
		\end{align*}
		Optimality, \eqref{eq:exactexpansion}, decomposability of the nuclear and $\ell_1$ norms, and \eqref{eq:score_event} give
		\[\frac12\|H\|_n^2+\frac12P_\perp \le\frac32P_T+|\langle R_n,H\rangle|.\]
		Using \eqref{eq:remainder_event},
		\begin{equation}
			\frac38\|H\|_n^2+\frac12P_\perp \le\frac32P_T+B_n. \label{eq:joint_basic_final}
		\end{equation}
		This is one weighted inequality.
		
		If $B_n\le P_T/4$, dropping the quadratic term in \eqref{eq:joint_basic_final} gives $P_\perp\le(7/2)P_T<4P_T$, hence $(H_L,H_S)$ belongs to the joint cone. Since $\rank(P_\mathcal T H_L)\le2r$ and $P_\Omega H_S$ has at most $s$ nonzero entries,
		\begin{align*}
			P_T&\le\lambda_*\sqrt{2r}\|H_L\|_F+\lambda_1\sqrt{s}\|H_S\|_F\\
			&\le\sqrt{2r\lambda_*^2+s\lambda_1^2} \bigl(\|H_L\|_F^2+\|H_S\|_F^2\bigr)^{1/2}\\
			&\le\alpha^{-1}\sqrt{2r\lambda_*^2+s\lambda_1^2}\|H\|_F.
		\end{align*}
		Applying the lower covariance bound to \eqref{eq:joint_basic_final} yields
		\[\frac{3\kappa}{8}\|H\|_F^2\le\frac74P_T,\]
		and therefore
		\[\|H\|_F^2\le\frac{196}{9\kappa^2\alpha^2} (2r\lambda_*^2+s\lambda_1^2).
		\]
		
		If $B_n>P_T/4$, then $P_T<4B_n$. Equation \eqref{eq:joint_basic_final} gives
		\[\frac{3\kappa}{8}\|H\|_F^2 \le7B_n,\]
		so $\|H\|_F^2\le56B_n/(3\kappa)$. Combining the two cases proves \eqref{eq:det_exact}.
	\end{proof}
	
	The preceding case split replaces any need to assume that a generic positive bias is eventually dominated by a stochastic term.
	
	\begin{theorem}\label{thm:det_oracle}
		Let $\bar A=\bar L+\bar S$ be any comparator with $\rank(\bar L)\le r$, $\|\bar S\|_0\le s$, satisfying Assumption~\ref{ass:angle}. On $\mathcal E(\kappa,K,B_n)$,
		\begin{equation}
			\begin{split}
				\|\widehat A-A_0\|_F^2
				\le{}&
				C_1\left(1+\frac K\kappa\right)\|\bar A-A_0\|_F^2
				+\frac{C_2}{\kappa}B_n\\
				&+\frac{C_3}{\kappa^2\alpha^2}
				\left(2r\lambda_*^2+s\lambda_1^2\right),
			\end{split}
			\label{eq:oracle_approx}
		\end{equation}
		for universal constants $C_1,C_2,C_3>0$. Consequently, the right-hand side may be minimized over admissible $(r,s,\bar L,\bar S)$.
	\end{theorem}
	
	\begin{proof}
		Set $E=\bar A-A_0$ and $H=(\widehat L-\bar L)+(\widehat S-\bar S)$. Comparing the objective at $(\widehat L,\widehat S)$ and $(\bar L,\bar S)$ and expanding the loss around $A_0$ introduces the additional covariance inner product $\langle E,H\rangle_n$. By Young's inequality,
		\[|\langle E,H\rangle_n| \le\frac18\|H\|_n^2+2\|E\|_n^2.\]
		Together with \eqref{eq:remainder_event}, the analogue of \eqref{eq:joint_basic_final} is
		\[\frac14\|H\|_n^2+\frac12P_\perp \le\frac32P_T+B_n+2\|E\|_n^2.\]
		Apply the same two-case argument with $B_n+2\|E\|_n^2$ in place of $B_n$, use $\|E\|_n^2\le K\|E\|_F^2$, and finally $\|\widehat A-A_0\|_F^2\le2\|H\|_F^2+2\|E\|_F^2$. Collecting universal constants gives \eqref{eq:oracle_approx}.
	\end{proof}
	
	Theorem~\ref{thm:det_oracle} permits both singular-value tails and small off-support entries through the approximation term. This parallels the approximate-structure perspective in noisy matrix decomposition \citep{AgarwalNegahbanWainwright2012}, while the probabilistic verification below is specific to discretely observed L\'evy--OU dynamics.
	
	\section{Probabilistic verification for the L\'evy--OU loss} \label{sec:probabilistic}
	
	\subsection{Exact innovation representation and concentration scale}
	
	The exact OU transition over one sampling interval is
	\begin{equation}
		\Delta X_k
		=\bigl(e^{-\Delta_nA_0}-I_d\bigr)X_{k-1}+\Delta\widetilde Z_k,
		\qquad \Delta\widetilde Z_k
		:=\int_{t_{k-1}}^{t_k}e^{-(t_k-u)A_0}\,dZ_u.\label{eq:innovation}\end{equation}
	The innovations $\Delta\widetilde Z_k$ are i.i.d. and independent of $\mathcal F_{t_{k-1}}$. Define
	\begin{equation}
		\gamma(\Delta_n) :=\lambda_{\max}(C+\nu_2) \left(\lambda_{\max}(C_\infty(B))\vee1\right)
		\exp\!\left(\Delta_n\|A_0\|_{\mathrm{op}}\right).
		\label{eq:gamma}
	\end{equation}
	This is the concentration scale used in the discrete sparse L\'evy--OU theory \citep[Theorem~3.1]{DexheimerJeszka2026}. For fixed model and localization constants it remains bounded as $\Delta_n\downarrow0$.
	
	The increment filter and the innovation filter differ by the deterministic drift displacement. Following \citet[Lemma~A.2]{DexheimerJeszka2026}, set
	\begin{equation}
		\eta_\pm
		:=\eta\pm b\left(e^{\Delta_n\|A_0\|_{\mathrm{op}}}-1\right).
		\label{eq:etapm}
	\end{equation}
	The assumption
	\begin{equation}
		\eta\ge2b\left(e^{\Delta_n\|A_0\|_{\mathrm{op}}}-1\right)
		\label{eq:eta_condition}
	\end{equation}
	implies $\eta_-\ge\eta/2>0$.
	
	\subsection{The centered score: operator and infinity norms}
	
	Using \eqref{eq:innovation}, the stochastic part of the score is obtained by centering the filtered innovation term conditionally on $\mathcal F_{t_{k-1}}$. Let
	\begin{equation}
		M_n
		:=\frac1T\sum_{k=1}^n
		\mathbbm1_{\{X_{k-1}\in B\}}
		\left[
		\Delta\widetilde Z_k\mathbbm1_{\{\|\Delta X_k\|\le\eta\}}
		-\mathbb E\!\left(\Delta\widetilde Z_k\mathbbm1_{\{\|\Delta X_k\|\le\eta\}}\mid\mathcal F_{t_{k-1}}\right)
		\right]X_{k-1}^\top.
		\label{eq:Mn}
	\end{equation}
	For a set $\mathcal D\subset\mathbb R^{d\times d}$, Theorem~4.8 of \citet{DexheimerJeszka2026} gives, on their quadratic-variation event $Q_T$, a uniform bound for the centered stochastic process indexed by $\mathcal D$ of order
	\begin{equation}
		\sqrt{\frac{\gamma(\Delta_n)}{T}}
		\left\{w(\mathcal D)+\sqrt{\log(2/\delta)}\,\rad(\mathcal D)\right\},
		\label{eq:DJ_width}
	\end{equation}
	where $w(\mathcal D)$ is Gaussian width and $\rad(\mathcal D)=\sup_{H\in\mathcal D}\|H\|_F$.
	
	\begin{lemma}\label{lem:dual}
		There is a universal constant $C_S>0$ such that, whenever the sample-complexity conditions ensuring the event $Q_T$ hold, with probability at least $1-\delta_S$,
		\begin{align}
			\|M_n\|_{\mathrm{op}}
			&\le C_S\sqrt{\frac{\gamma(\Delta_n)}T}
			\left(2\sqrt d+\sqrt{\log(4/\delta_S)}\right),
			\label{eq:opscore}\\
			\|M_n\|_\infty
			&\le C_S\sqrt{\frac{\gamma(\Delta_n)}T}
			\left(\sqrt{2\log(2d^2)}+\sqrt{\log(4/\delta_S)}\right).
			\label{eq:infscore}
		\end{align}
	\end{lemma}
	
	\begin{proof}
		For the operator norm take
		\[\mathcal D_{\mathrm{op}}=\{uv^\top:\|u\|_2=\|v\|_2=1\}.\]
		Then $\sup_{H\in\mathcal D_{\mathrm{op}}}\langle M_n,H\rangle=\|M_n\|_{\mathrm{op}}$, $\rad(\mathcal D_{\mathrm{op}})=1$, and for a standard Gaussian $d\times d$ matrix $G$,
		\[w(\mathcal D_{\mathrm{op}})=\mathbb E\|G\|_{\mathrm{op}}\le2\sqrt d.\]
		Substitution into \eqref{eq:DJ_width} gives \eqref{eq:opscore}. This $\sqrt d$ Gaussian width is the source of the $rd/T$ low-rank complexity.
		
		For the entrywise maximum norm use
		\[\mathcal D_\infty=\{\pm E_{ij}:1\le i,j\le d\}.\]
		Its radius is one and
		\[w(\mathcal D_\infty)=\mathbb E\max_{i,j}|G_{ij}| \le\sqrt{2\log(2d^2)}.\]
		Equation \eqref{eq:infscore} follows from the same theorem.
	\end{proof}
	
	Accordingly, it is sufficient to choose
	\begin{align}
		\lambda_* &=2C_S\sqrt{\frac{\gamma(\Delta_n)}T}
		\left(2\sqrt d+\sqrt{\log(4/\delta_S)}\right), \label{eq:lambdastar}\\
		\lambda_1
		&=2C_S\sqrt{\frac{\gamma(\Delta_n)}T} \left(\sqrt{2\log(2d^2)}+\sqrt{\log(4/\delta_S)}\right). \label{eq:lambdaone}
	\end{align}
	
	\subsection{Curvature of the increment-truncated covariance}
	
	The covariance result required here is precisely the increment-truncated result, not concentration around an incorrectly matched population covariance. Proposition~4.5 of \citet{DexheimerJeszka2026} combines localized covariance concentration with a matrix Freedman inequality. In the notation above, its consequence is that, for $T$ exceeding the corresponding finite-sample threshold, with probability at least $1-\delta_C$,
	\begin{align}
		\lambda_{\min}(\widehat C_n^{B,\eta})
		&\ge\frac14\lambda_{\min}(C_\infty(B))
		\mathbb P(\|\Delta\widetilde Z_1\|\le\eta_-),
		\label{eq:cov_lower}\\
		\lambda_{\max}(\widehat C_n^{B,\eta})
		&\le2\lambda_{\max}(C_\infty(B))
		\mathbb P(\|\Delta\widetilde Z_1\|\le\eta_+).
		\label{eq:cov_upper}
	\end{align}
	The displayed constants are convenient consequences of the events in that proposition; changing fixed numerical constants has no effect on the rates below.
	
	We use the truncation condition
	\begin{assumption}[Tail-filter condition]\label{ass:N}
		There is a constant $c_\eta>0$ such that
		\begin{equation}
			\mathbb P\left(\|\Delta\widetilde Z_1\|\ge\eta/2\right)
			\le c_\eta\frac{d\Delta_n}{T}.
			\label{eq:N}
		\end{equation}
	\end{assumption}
	If \eqref{eq:eta_condition}, Assumption~\ref{ass:N}, and $T\ge2c_\eta d\Delta_n$ hold, then
	$\mathbb P(\|\Delta\widetilde Z_1\|\le\eta_-)\ge1/2$. Hence \eqref{eq:cov_lower}--\eqref{eq:cov_upper} give the explicit choices
	\begin{equation}
		\kappa=\frac18\lambda_{\min}(C_\infty(B)),
		\qquad
		K=2\lambda_{\max}(C_\infty(B)).
		\label{eq:kappaK}
	\end{equation}
	In particular, curvature holds for every matrix direction $H$ once the covariance event occurs.
	
	This also makes the dimensional requirement transparent: $\widehat C_n^{B,\eta}$ has rank at most the number of retained design vectors. Positive-definite curvature therefore requires at least $d$ linearly independent retained observations. The finite-sample covariance conditions below guarantee this in the admissible high-frequency regime.
	
	\subsection{Discretization and truncation remainders}
	
	Let
	\begin{equation}
		R_\Delta:=e^{-\Delta_nA_0}-I_d+\Delta_nA_0.
		\label{eq:Rdelta}
	\end{equation}
	By \eqref{eq:innovation}, the non-centered part of the score consists of a deterministic Euler remainder involving $R_\Delta X_{k-1}$ and the conditional mean introduced by increment truncation. Proposition~4.2 and Lemma~4.6 of \citet{DexheimerJeszka2026} give uniform Young-inequality bounds of exactly the form needed in \eqref{eq:remainder_event}. To make the quantities explicit, define
	\begin{equation}
		\mathrm{Disc}_n
		:=C_D K\Delta_n^2\|A_0\|_F^4
		\left(1+\Delta_n\|A_0\|_F e^{\Delta_n\|A_0\|_F}\right)^2,
		\label{eq:disc}
	\end{equation}
	where $C_D$ is universal, and
	\begin{equation}
		\mathrm{Trunc}_n
		:=C_J\frac{\lambda_{\max}(C+\nu_2)e^{2\Delta_n\|A_0\|_{\mathrm{op}}}}{\Delta_n}
		\mathbb P\left(\|\Delta\widetilde Z_1\|>\eta_-\right),
		\label{eq:trunc}
	\end{equation}
	where $C_J$ is universal. Taylor's theorem gives
	$\|R_\Delta\|_F\lesssim\Delta_n^2\|A_0\|_F^2(1+\Delta_n\|A_0\|_Fe^{\Delta_n\|A_0\|_F})$, producing \eqref{eq:disc}. Conditional centering and Young's inequality applied to the removed innovation event produce \eqref{eq:trunc}. Consequently, after adjusting fixed constants,
	\begin{equation}
		|\langle R_n,H\rangle|
		\le\frac18\|H\|_n^2
		+\mathrm{Disc}_n+\mathrm{Trunc}_n
		\qquad\text{for all }H.
		\label{eq:remainder_verified}
	\end{equation}
	Thus Definition~\ref{def:good} holds with
	\begin{equation}
		B_n=\mathrm{Disc}_n+\mathrm{Trunc}_n.
		\label{eq:Bn}
	\end{equation}
	
	If $\|A_0\|_F^2=O(d)$, $K=O(1)$, and $\Delta_n\|A_0\|_F=O(1)$, then
	\begin{equation}
		\mathrm{Disc}_n=O(d^2\Delta_n^2).
		\label{eq:disc_order}
	\end{equation}
	Moreover, \eqref{eq:eta_condition} gives $\eta_-\ge\eta/2$, so Assumption~\ref{ass:N} implies
	\begin{equation}
		\mathrm{Trunc}_n=O(d/T)
		\label{eq:trunc_order}
	\end{equation}
	for uniformly bounded noise and stability constants.
	
	\subsection{Explicit L\'evy-regime sample complexity}
	
	The required observation horizon is inherited from the same covariance, quadratic-variation, and truncation events used in \citet{DexheimerJeszka2026}. Let $T_{\star,0}(\delta)$ denote their base threshold determined by localized covariance concentration and mixing. Up to model-dependent constants and lower-order logarithms, their regime-specific choices can be summarized as follows.
	
	When invoking Proposition~4.11 of Dexheimer and Jeszka, we also impose its polynomial sampling-growth condition $n\le c_{\mathrm{sg}}T^{\delta_{\mathrm{sg}}}$ for fixed constants $c_{\mathrm{sg}},\delta_{\mathrm{sg}}>0$.
	
	\begin{table}[ht]
		\centering
		\caption{Orders of truncation level and observation-horizon threshold inherited from the discrete L\'evy--OU concentration theory.}
		\label{tab:regimes}
		\begin{tabular}{lll}
			\toprule
			BDLP regime & order of $\eta$ & order of $T_\star(\delta,\eta)$ \\
			\midrule
			continuous
			& $\sqrt{d\Delta_n\log T}$
			& $d^2\Delta_n\log T\log(1/\delta)\vee T_{\star,0}$ \\
			bounded jumps
			& $\sqrt d\,\log T$
			& $d^2\log T\log(1/\delta)\vee T_{\star,0}$ \\
			sub-Weibull $\alpha$
			& $\sqrt d\,(\log T)^{1+1/\alpha}$
			& $d^2(\log T)^{2+2/\alpha}\log(1/\delta)\vee T_{\star,0}$ \\
			$p$th moment
			& $T^{1/p}d^{1/2-1/p}$
			& $d^{2-2/p}T^{2/p}\log(1/\delta)\vee T_{\star,0}$ \\
			\bottomrule
		\end{tabular}
	\end{table}
	
	Table~\ref{tab:regimes} is the order-level content of Table~1 and Proposition~4.11 in \citet{DexheimerJeszka2026}. In the continuous case their mixing analysis gives $T_{\star,0}$ of order $d\log^2(d/\delta)$ under the corresponding uniform covariance assumptions. The key point for the present paper is that the addition of low-rank regularization changes the score geometry but does not alter the L\'evy-tail mechanism producing these sample-complexity thresholds.
	
	\section{Main oracle inequalities}
	\label{sec:main}
	
	We now combine the deterministic geometry with the probabilistic ingredients. Allocate a total failure probability $\delta$ as $\delta_S+\delta_C+\delta_R\le\delta$; for concreteness one may take all three equal to $\delta/3$. The score, covariance, and remainder events are then combined by a union bound.
	
	\begin{theorem}\label{thm:main_exact}
		Suppose Assumptions~\ref{ass:A0}, \ref{ass:angle}, and \ref{ass:N} hold for an exact decomposition $A_0=L_0+S_0$ with $\rank(L_0)\le r$ and $\|S_0\|_0\le s$. Suppose \eqref{eq:eta_condition} holds and $T$ exceeds the regime-specific threshold ensuring the score, covariance, and truncation events. Choose $\lambda_*$ and $\lambda_1$ as in \eqref{eq:lambdastar}--\eqref{eq:lambdaone}. Then, with probability at least $1-\delta$,
		\begin{align}
			\|\widehat A-A_0\|_F^2
			\le{}&C\{\mathrm{Disc}_n+\mathrm{Trunc}_n\}\nonumber\\
			&+C\frac{\gamma(\Delta_n)}T
			\left[r\{d+\log(1/\delta)\}+s\{\log(2d^2)+\log(1/\delta)\}\right],
			\label{eq:main_exact}
		\end{align}
		where $C$ depends only on uniform lower bounds for $\lambda_{\min}(C_\infty(B))$ and $\alpha$, upper bounds for $\lambda_{\max}(C_\infty(B))$, and fixed L\'evy/localization constants.
	\end{theorem}
	
	\begin{proof}
		On the intersection of the three good events, \eqref{eq:kappaK}, Lemma~\ref{lem:dual}, and \eqref{eq:remainder_verified} verify Definition~\ref{def:good}. Apply Theorem~\ref{thm:det_exact}. Squaring \eqref{eq:lambdastar} and \eqref{eq:lambdaone} gives
		\[	r\lambda_*^2\lesssim \frac{\gamma(\Delta_n)}T r\{d+\log(1/\delta)\},\] and	\[s\lambda_1^2\lesssim \frac{\gamma(\Delta_n)}T s\{\log(2d^2)+\log(1/\delta)\}. \]
		Absorb the fixed powers of $\kappa^{-1}$ and $\alpha^{-1}$ into $C$.
	\end{proof}
	
	\begin{corollary}\label{cor:standard}
		Under the standard tail regimes in Table~\ref{tab:regimes}, with uniformly bounded model constants, $\|A_0\|_F^2=O(d)$, and $\Delta_n\|A_0\|_F=O(1)$, Theorem~\ref{thm:main_exact} gives, ignoring confidence logarithms,
		\begin{equation}
			\|\widehat A-A_0\|_F^2 \lesssim
			d^2\Delta_n^2+\frac dT	+\frac{\gamma(\Delta_n)}T\{rd+s\log d\}.
			\label{eq:standard_rate}
		\end{equation}
	\end{corollary}
	
	The $rd$ term in \eqref{eq:standard_rate} has the order of the effective number of free parameters in a rank-$r$, $d\times d$ matrix. The result is an upper oracle bound for the combined L\'evy--OU model; no minimax-optimality claim for the combined low-rank-plus-sparse class is required.
	
	\begin{theorem}\label{thm:main_oracle}
		Under the probabilistic conditions of Theorem~\ref{thm:main_exact}, let $\mathfrak A$ be any family of comparators $\bar A=\bar L+\bar S$ satisfying Assumption~\ref{ass:angle}. With probability at least $1-\delta$,
		\begin{equation}
			\begin{split}
				\|\widehat A-A_0\|_F^2 \le C\inf_{\bar A\in\mathfrak A} \Bigg[&\|\bar A-A_0\|_F^2 +\mathrm{Disc}_n+\mathrm{Trunc}_n\\
				&+\frac{\gamma(\Delta_n)}T \left\{r(\bar L)d+s(\bar S)\log d\right\} \Bigg],
			\end{split}	\label{eq:approx_main}
		\end{equation}
		up to confidence logarithms and uniform structural/covariance constants.
	\end{theorem}
	
	\begin{proof}
		Apply Theorem~\ref{thm:det_oracle} to each admissible comparator and substitute the verified penalty scales, covariance bounds, and remainder terms. Taking the infimum yields \eqref{eq:approx_main}.
	\end{proof}
	
	Theorem~\ref{thm:main_oracle} gives a direct bias--complexity tradeoff: singular values or entries that are too small to estimate can be left in the approximation term instead of being forced into exact rank or support constraints.
	
	\begin{corollary}\label{cor:consistency}
		Consider a sequence of models for which $\gamma(\Delta_n)$, the covariance condition numbers, and the restricted cancellation constants are uniformly bounded. If
		\begin{equation}
			d^2\Delta_n^2\to0, \qquad
			\frac{d+rd+s\log d}{T}\to0,\label{eq:consistency_conditions}
		\end{equation}
		then the exact-structure estimator satisfies $\|\widehat A-A_0\|_F\to0$ in probability. For approximate structure, the same conclusion holds whenever there are admissible comparators whose Frobenius approximation error tends to zero and whose $(r,s)$ satisfy \eqref{eq:consistency_conditions}. \end{corollary}
	
	\section{Computation and tuning} \label{sec:computation}
	
	The objective in \eqref{eq:estimator} is convex. Let
	\[C_n:=\widehat C_n^{B,\eta},\qquad D_n^{\mathrm{obs}}:=\frac1T\sum_{k=1}^n I_k\Delta X_kX_{k-1}^\top. \]Then \[\nabla\mathcal L_n^D(A)=D_n^{\mathrm{obs}}+AC_n.\] Starting from $L^{(0)}=S^{(0)}=0$, set $A^{(m)}=L^{(m)}+S^{(m)}$ and \[G^{(m)}=D_n^{\mathrm{obs}}+A^{(m)}C_n.\]
	For
	\begin{equation}
		0<\tau<\frac1{2\|C_n\|_{\mathrm{op}}},\label{eq:stepsize}
	\end{equation}
	the product-space proximal-gradient updates are
	\begin{align}
		L^{(m+1)} &=\operatorname{SVT}_{\tau\lambda_*} \left(L^{(m)}-\tau G^{(m)}\right),\label{eq:svt}\\
		S^{(m+1)} &=\operatorname{soft}_{\tau\lambda_1}
		\left(S^{(m)}-\tau G^{(m)}\right),\label{eq:soft}
	\end{align}
	where $\operatorname{SVT}$ thresholds singular values and $\operatorname{soft}$ is entrywise soft thresholding. The factor two in \eqref{eq:stepsize} is the Lipschitz factor for the smooth loss viewed as a function of the product variable $(L,S)$.
	
	The theory determines the relative penalty scales
	\begin{equation}
		\lambda_*\asymp\sqrt{\frac{\gamma(\Delta_n)d}{T}},
		\qquad
		\lambda_1\asymp\sqrt{\frac{\gamma(\Delta_n)\log(d^2)}{T}}.
		\label{eq:tuning_scales}
	\end{equation}
	For applications, unknown multiplicative constants can be selected by chronological validation. In the simulations, let
	\[
	\widehat\sigma
	:=\left\{
	\frac{\operatorname{median}_k\|\Delta X_k\|_2^2}
	{d\Delta_n}
	\right\}^{1/2},
	\qquad
	b_*:=\widehat\sigma\sqrt{\frac dT},
	\qquad
	b_1:=\widehat\sigma
	\sqrt{\frac{2\log(2d^2)}T}.
	\]
	We parameterize the two penalties by a common scale $c$ and a nuclear-to-sparse multiplier $q$,
	\begin{equation}
		\lambda_*=qc\,b_*,
		\qquad
		\lambda_1=c\,b_1,
		\label{eq:cv_parameterization}
	\end{equation}
	with
	\[
	c\in\{0.10,0.15,0.20,0.25,0.30,0.40\},
	\qquad
	q\in\{2,3,4,5,6\}.
	\]
	This parameterization separates overall regularization strength from the relative weighting of the two penalties. A chronological 75/25 training/validation split selects $(c,q)$. The localization radius and increment threshold are empirical $0.98$ and $0.95$ quantiles, respectively, of the corresponding Euclidean norms on the training segment; the selected multipliers are then used in the full-path refit with the same quantile rules. The localized Lasso uses the same split and filter, with multiplier grid
	$\{0.15,0.30,0.60,1.00,1.50,2.20\}$ applied to $b_1$.
	
	The proximal-gradient iterations use the step size $0.95/(2\|C_n\|_{\mathrm{op}})$ for the combined estimator and $0.95/\|C_n\|_{\mathrm{op}}$ for Lasso. Iteration stops when the relative change is below $10^{-6}$, with a cap of 1000 iterations. In the final simulation study every full-path fit satisfies this stopping criterion.
	
	\section{Simulation study}
	\label{sec:simulation}
	
	The numerical study evaluates estimation of the total drift $A_0$ under an exact low-rank-plus-sparse truth, examines how performance changes with the observation horizon at a fixed sampling mesh, and checks adaptation to a sparse-only truth. All comparisons are paired path-by-path, and each design uses 100 independent paths. The complete simulation code accompanies the manuscript.
	
	\subsection{Data-generating process}
	
	We set $d=30$ and $\Delta_n=0.05$. The mixed truth is
	\begin{equation}
		A_0=L_0+S_0,
		\qquad
		L_0=Q\operatorname{diag}(1,0.7)Q^\top,
		\label{eq:sim_truth}
	\end{equation}
	where $Q\in\mathbb R^{30\times2}$ has orthonormal columns obtained from a Gaussian matrix. The sparse component has diagonal $0.65I_d$ together with 15 symmetric off-diagonal edges. Each edge has an independent random sign and magnitude drawn uniformly from $[0.035,0.075]$. The sparse-only benchmark sets $L_0=0$ and uses an independently generated sparse matrix from the same construction. A single realized truth is fixed within each design and used for all 100 replicates. The realized drift matrices are symmetric positive definite.
	
	The BDLP is the sum of isotropic Brownian noise with diffusion standard deviation $0.7$ and an independent compound-Poisson process with intensity $0.3$ and $N(0,I_d)$ jump sizes. Paths are generated by the exact OU recursion: the deterministic transition uses $e^{-\Delta_nA_0}$, the Brownian increment uses its exact integrated covariance, and each jump is propagated from its simulated within-interval arrival time. We discard 1200 burn-in increments.
	
	For the mixed truth we use
	\[
	(T,n)\in\{(50,1000),(100,2000),(200,4000)\},
	\]
	while the sparse-only benchmark uses $(T,n)=(100,2000)$. The proposed low-rank-plus-sparse estimator and localized Lasso are fitted to the same path and use the tuning procedure in Section~\ref{sec:computation}. The principal loss is relative Frobenius error,
	\begin{equation}
		\operatorname{RelErr}(\widehat A)
		:=\frac{\|\widehat A-A_0\|_F}{\|A_0\|_F}.
		\label{eq:relative_error_sim}
	\end{equation}
	The average retained fraction after localization and increment truncation is approximately $0.931$ in every design.
	
	\subsection{Total-drift estimation}
	
	\begin{table}[ht]
		\centering
		\caption{Total-drift estimation over 100 paired Monte Carlo paths. Lasso and low-rank-plus-sparse entries are mean relative Frobenius error (standard deviation). The paired difference is low-rank-plus-sparse minus Lasso, and the win rate is the percentage of paths on which this difference is negative.}
		\label{tab:main_simulation}
		\small
		\begin{tabular}{llcccc}
			\toprule
			Truth & $T$ & Lasso & LR+$S$ & Paired difference & Win rate \\
			\midrule
			Mixed & 50
			& 0.502 (0.026) & \textbf{0.474 (0.026)} & $-0.0282$ & 89\% \\
			Mixed & 100
			& 0.423 (0.017) & \textbf{0.404 (0.018)} & $-0.0187$ & 92\% \\
			Mixed & 200
			& 0.373 (0.011) & \textbf{0.351 (0.017)} & $-0.0222$ & 89\% \\
			Sparse only & 100
			& 0.310 (0.022) & \textbf{0.299 (0.018)} & $-0.0113$ & 75\% \\
			\bottomrule
		\end{tabular}
	\end{table}
	
	Table~\ref{tab:main_simulation} shows a stable aggregate-estimation advantage for the combined penalty. Relative to Lasso, mean relative Frobenius error is lower by approximately $5.6\%$, $4.4\%$, and $5.9\%$ at $T=50,100,200$ under the mixed truth. The corresponding paired win rates are $89\%$, $92\%$, and $89\%$. Approximate 95\% normal-approximation Monte Carlo intervals for the mean paired differences are $[-0.0319,-0.0244]$, $[-0.0208,-0.0166]$, and $[-0.0250,-0.0194]$, respectively. Thus the improvement is systematic across paths rather than being driven by a small number of favorable realizations.
	
	The sparse-only benchmark gives mean errors $0.310$ for Lasso and $0.299$ for the combined estimator, a reduction of approximately $3.6\%$, with a $75\%$ paired win rate and a corresponding approximate 95\% normal-approximation interval $[-0.0146,-0.0080]$ for the mean paired difference. In this sparse-only benchmark, the additional nuclear penalty does not degrade aggregate estimation on average and yields a modest improvement. Since $\Delta_n$ is fixed throughout, the sequence $T=50,100,200$ is interpreted as a horizon diagnostic rather than an estimate of a high-frequency convergence exponent. For both methods the mean error decreases monotonically as the observed horizon grows.
	
	\subsection{Component and rank diagnostics}
	
	The oracle inequalities in Sections~\ref{sec:deterministic}--\ref{sec:main} control the total drift $\widehat A-A_0$. To distinguish this target from exact decomposition recovery, we also record errors for $\widehat L$ and $\widehat S$, numerical-rank summaries, and the distance between the leading low-rank subspaces. For the mixed truth the true low-rank component has rank two. Let $\widehat L_{(2)}$ denote the best rank-two truncation of $\widehat L$, and define the normalized projector error
	\[
	\frac{\|\widehat P_2-P_2\|_F}{\sqrt{4}},
	\]
	where $P_2$ and $\widehat P_2$ project onto the true and estimated leading two-dimensional left singular subspaces.
	
	\begin{table}[ht]
		\centering
		\caption{Component diagnostics for the low-rank-plus-sparse estimator over 100 paths. Error columns report means; numerical-rank columns report medians. Rank$_{10^{-3}}$ counts singular values of $\widehat L$ exceeding $10^{-3}$, and Rank$_{95\%}$ is the number of singular values required to explain 95\% of $\|\widehat L\|_F^2$.}
		\label{tab:component_diagnostics}
		\small
		\begin{tabular}{llcccccc}
			\toprule
			Truth & $T$ & $\operatorname{rank}(L_0)$ & Rank$_{10^{-3}}$ & Rank$_{95\%}$
			& Rel. $L$ err. & Rank-2 $L$ err. & Rel. $S$ err. \\
			\midrule
			Mixed & 50 & 2 & 2 & 2 & 1.032 & 1.024 & 0.369 \\
			Mixed & 100 & 2 & 2 & 2 & 1.005 & 1.004 & 0.282 \\
			Mixed & 200 & 2 & 5 & 3 & 0.956 & 0.955 & 0.217 \\
			Sparse only & 100 & 0 & 1 & 1 & -- & -- & 0.292 \\
			\bottomrule
		\end{tabular}
	\end{table}
	
	The component diagnostics sharpen the interpretation of Table~\ref{tab:main_simulation}. Sparse-component recovery improves steadily with the horizon: the mean relative error of $\widehat S$ decreases from $0.369$ at $T=50$ to $0.282$ at $T=100$ and $0.217$ at $T=200$. Low-rank identification is substantially more difficult. The mean relative error of $\widehat L$ is $1.032$, $1.005$, and $0.956$, and truncating $\widehat L$ to the true rank two changes these values only to $1.024$, $1.004$, and $0.955$. The mean normalized subspace-projector errors are $0.963$, $0.954$, and $0.859$. Thus the leading estimated low-rank directions do not recover the true subspace accurately even when the total drift is estimated more accurately than by Lasso.
	
	Numerical rank displays the same distinction. At $T=50$ and $T=100$, the median $10^{-3}$-threshold rank and the median 95\%-energy rank both equal the true rank two, but at $T=200$ they increase to five and three. In the sparse-only design, where the true low-rank component has rank zero, both corresponding medians equal one. These results show that the empirical advantage of the combined penalty is an advantage in estimating $A_0=L_0+S_0$; it should not be interpreted as exact recovery of the decomposition without additional singular-value, minimum-signal, and separation conditions.
	
	\subsection{Tuning and numerical stability}
	
	The scale--ratio parameterization in \eqref{eq:cv_parameterization} produces a well-dispersed validation profile in the final experiment. The common scale $c=0.25$ is selected in $46\%$, $58\%$, and $56\%$ of mixed-truth paths at $T=50,100,200$, and in $67\%$ of sparse-only paths. Neither boundary value $c=0.10$ nor $c=0.40$ is selected in the final study. The multiplier ratio is concentrated primarily on $q=3$ and $q=4$, which together account for $70\%$, $64\%$, $75\%$, and $61\%$ of selections in the four designs. Selection of the upper ratio boundary $q=6$ remains limited to $15\%$, $15\%$, $11\%$, and $18\%$, respectively. The validation grid therefore contains the empirically preferred region in its interior with respect to overall regularization strength and with only modest use of the ratio boundary.
	
	Every final full-path optimization converges under the tolerance $10^{-6}$. For the combined estimator, the mean iteration counts are $256.4$, $159.7$, $140.6$, and $147.7$ across the four designs, with a maximum of 501 iterations; for Lasso the maximum is 221. The numerical results are therefore obtained from converged final fits rather than iteration-capped solutions.
	
	\section{Discussion}
	
	The low-rank-plus-sparse formulation extends high-dimensional L\'evy--OU drift estimation to dependence structures that combine dense common modes with localized direct interactions. The proof separates the problem into stochastic-process concentration and structured convex geometry. On the stochastic side, the discrete L\'evy--OU machinery supplies increment-truncated covariance concentration, a Gaussian-width martingale inequality, deterministic discretization control, and regime-specific tail filtering. On the structural side, nuclear and $\ell_1$ penalties produce one weighted joint cone, and restricted non-cancellation on that cone converts component complexity into total-drift Frobenius risk.
	
	Three features sharpen the resulting theory. First, the operator-norm score has Gaussian width of order $\sqrt d$, yielding the natural $rd/T$ low-rank contribution. Second, Theorem~\ref{thm:main_oracle} permits approximate rather than exact low-rank-plus-sparse structure. Third, the covariance and sample-complexity assumptions are explicit and use the same increment-truncated empirical matrix that appears in the loss.
	
	The final Monte Carlo study aligns directly with the total-drift target of the oracle inequalities. Under the mixed rank-two-plus-sparse truth, the combined estimator lowers mean relative Frobenius error by $4.4$--$5.9\%$ relative to localized Lasso across horizons $T=50,100,200$ and wins $89$--$92\%$ of paired paths. The improvement persists in the sparse-only benchmark, where the combined estimator wins $75\%$ of paths. Chronological validation selects an interior overall regularization scale, and every final full-path optimization satisfies the prescribed convergence tolerance.
	
	The component diagnostics also identify the statistical task on which the combined regularizer is most effective. Recovery of the sparse component improves strongly with the observation horizon, whereas the low-rank component, its numerical rank, and its leading singular subspace remain much less accurately identified. In particular, imposing the true rank two after estimation produces little reduction in low-rank Frobenius error. The observed gain is therefore a gain in estimation of the aggregate drift $A_0$, not evidence of exact latent decomposition. This distinction is consistent with Theorems~\ref{thm:main_exact} and \ref{thm:main_oracle}, whose loss target is $\|\widehat A-A_0\|_F$.
	
	Further extensions include minimax lower bounds for the combined L\'evy--OU model, component-recovery theory under explicit singular-gap and minimum-signal conditions, fully data-driven theory for $B$, $\eta$, and the penalty constants, and inferential results for selected entries or low-dimensional functionals of the drift.
	
	\section*{Conflict of Interest}
	The author declares that there is no conflict of interest.


\begin{thebibliography}{99}
		\bibitem[Agarwal et al.(2012)]{AgarwalNegahbanWainwright2012}
		A.~Agarwal, S.~Negahban, and M.~J. Wainwright,
		\newblock Noisy matrix decomposition via convex relaxation: Optimal rates in high dimensions,
		\newblock \emph{Annals of Statistics} \textbf{40} (2012), no.~2, 1171--1197.
		
		\bibitem[Chandrasekaran et al.(2011)]{ChandrasekaranSanghaviParriloWillsky2011}
		V.~Chandrasekaran, S.~Sanghavi, P.~A. Parrilo, and A.~S. Willsky,
		\newblock Rank-sparsity incoherence for matrix decomposition,
		\newblock \emph{SIAM Journal on Optimization} \textbf{21} (2011), no.~2, 572--596.
		
		\bibitem[Cio{\l}ek et al.(2020)]{CiolekMarushkevychPodolskij2020}
		G.~Cio{\l}ek, D.~Marushkevych, and M.~Podolskij,
		\newblock On Dantzig and Lasso estimators of the drift in a high dimensional Ornstein--Uhlenbeck model,
		\newblock \emph{Electronic Journal of Statistics} \textbf{14} (2020), no.~2, 4395--4420.
		
		\bibitem[Dexheimer and Jeszka(2026)]{DexheimerJeszka2026}
		N.~Dexheimer and N.~Jeszka,
		\newblock Sparse estimation for high-dimensional L\'evy-driven Ornstein--Uhlenbeck processes from discrete observations,
		\newblock arXiv:2603.06176, 2026.
		
		\bibitem[Dexheimer and Strauch(2024)]{DexheimerStrauch2024}
		N.~Dexheimer and C.~Strauch,
		\newblock On Lasso and Slope drift estimators for L\'evy-driven Ornstein--Uhlenbeck processes,
		\newblock \emph{Bernoulli} \textbf{30} (2024), no.~1, 88--116.
		
		\bibitem[Ga\"iffas and Matulewicz(2019)]{GaiffasMatulewicz2019}
		S.~Ga\"iffas and G.~Matulewicz,
		\newblock Sparse inference of the drift of a high-dimensional Ornstein--Uhlenbeck process,
		\newblock \emph{Journal of Multivariate Analysis} \textbf{169} (2019), 1--20.
		
		\bibitem[Masuda(2004)]{Masuda2004}
		H.~Masuda,
		\newblock On multidimensional Ornstein--Uhlenbeck processes driven by a general L\'evy process,
		\newblock \emph{Bernoulli} \textbf{10} (2004), no.~1, 97--120.
		
		\bibitem[Negahban et al.(2012)]{NegahbanRavikumarWainwrightYu2012}
		S.~Negahban, P.~Ravikumar, M.~J. Wainwright, and B.~Yu,
		\newblock A unified framework for high-dimensional analysis of $M$-estimators with decomposable regularizers,
		\newblock \emph{Statistical Science} \textbf{27} (2012), no.~4, 538--557.
		
		\bibitem[Sato and Yamazato(1984)]{SatoYamazato1984}
		K.~Sato and M.~Yamazato,
		\newblock Operator-self-decomposable distributions as limit distributions of processes of Ornstein--Uhlenbeck type,
		\newblock \emph{Stochastic Processes and their Applications} \textbf{17} (1984), 73--100.
		
	\end{thebibliography}
\end{document}